\documentclass[12pt]{article}
\usepackage{amssymb,amsfonts,amsmath,amsthm,indentfirst,color,graphicx}
\usepackage[top=3cm,bottom=3cm,left=2.5cm,right=2.5cm]{geometry}

\newtheorem{theorem}{Theorem}[section]
\newtheorem{corollary}{Corollary}
\newtheorem{lemma}[theorem]{Lemma}
\newtheorem{proposition}{Proposition}

\newtheorem{remark}{Remark}[section]

\title{\Large \bf Longtime dynamics of a semilinear Lam\'e system}

\author{
	\small \bf Lito Edinson Bocanegra-Rodr{\'i}guez \\
	\footnotesize Institute of Mathematical and Computer Sciences, University of S\~ao Paulo, 13566-560 S\~ao Carlos, SP, Brazil \medskip \\
	\small \bf To Fu Ma \\ 
	\footnotesize Department of Mathematics, University of Bras\'{\i}lia, 70910-900 Bras\'{\i}lia, DF, Brazil  \medskip \\
	\small \bf Paulo Nicanor Seminario-Huertas  \thanks{Corresponding author.} \\ 
\footnotesize  Academic Department of Mathematics, National University of Callao, Bellavista 07011, Callao, Peru \\ 
\footnotesize Department of Mathematics, University of Bras\'{\i}lia, 70910-900 Bras\'{\i}lia, DF, Brazil \medskip  \\
	\small \bf Marcio Antonio Jorge Silva \\
	\footnotesize Department of Mathematics, State University of Londrina, 86057-970 Londrina, PR, Brazil
}

\date{}

\begin{document}

\maketitle

\begin{abstract}
This paper is concerned with longtime dynamics of semilinear Lam\'e systems 
$$
\partial^2_t u - \mu \Delta u - (\lambda + \mu) \nabla {\rm div} u + \alpha \partial_t u + f(u) = b,
$$
defined in bounded domains of $\mathbb{R}^3$ with Dirichlet boundary condition. 
Firstly, we establish the existence of finite dimensional global attractors subjected to a critical forcing $f(u)$. 
Writing $\lambda + \mu$ as a positive parameter $\varepsilon$, we discuss some physical aspects of the limit 
case $\varepsilon \to 0$. Then, we show the upper-semicontinuity of attractors with respect to the parameter when $\varepsilon \to 0$. To our best knowledge, the analysis of attractors for dynamics of Lam\'e systems has not been studied before. 
\end{abstract}

\medskip


\medskip


\section{Introduction} \label{sec-introduction}
\setcounter{equation}{0}

The Lam\'e system is a classical model for isotropic elasticity. In three dimensions,   
it is given by 
\begin{eqnarray}
\left\lbrace
\begin{array}{ll}
\partial^2_{t} u - \mu \Delta u - (\lambda+\mu) \nabla {\rm div} u = 0 & \text{in} \,\; \Omega\times\mathbb{R}^{+}, \smallskip \\
u=0 &  \text{on} \,\; \partial \Omega\times\mathbb{R}^{+}, \smallskip  \\
u(0)=u_0, \,\; \partial_t u(0)=u_1 &  \text{in} \; \Omega,
\end{array} 
\right.
\label{1}
\end{eqnarray}
where $\Omega$ is a bounded domain of $\mathbb{R}^3$ with smooth boundary $\partial \Omega$, 
representing the elastic body in its rest configuration. Here, the vector $u=(u_1,u_2,u_3)$ denotes displacements  
and $\lambda,\mu$ are Lam\'e's constants with $\mu > 0$. In this model, the stress tensor is given by
\begin{equation}\label{ten}
\sigma(u)_{ij}=\lambda {\rm div} u \, \delta_{ij} + \mu\left(\frac{\partial u_i}{\partial x_j}+\frac{\partial u_j}{\partial x_i} \right). 
\end{equation}
We refer the reader to \cite{Achenbach,Ciarlet,Love,Teodorescu} for modeling aspects and \cite{Cerv2001,KK,Pujol} for some applications of vector waves. Later, we discuss the physical justification of taking limit $\lambda + \mu \to 0$.

\smallskip 

We note that the energy functional corresponding to the linear system \eqref{1} is given by 
$$
E_{\ell}(t) = \frac{1}{2}\int_{\Omega}\left( |\partial_t u|^2 + \mu |\nabla u |^2 + (\lambda + \mu) |{\rm div}u|^2 \right) \, dx,
$$
which is conservative since we have formally $\frac{d}{dt}E_{\ell}(t)=0$. This motivated several papers on such systems where the main feature is finding suitable damping and controllers in order to get 
stabilization and controllability, respectively. Let us recall some related results. 
The exponential stabilization of Lam\'e systems, defined in exterior domains of $\mathbb{R}^3$ with Dirichlet boundary, 
was studied by Yamamoto \cite{Yama}. Uniform stabilization by nonlinear boundary feedback was studied by Horn \cite{Horn}. 
Polynomial stabilization with interior localized damping was studied by Astaburuaga and Char\~ao \cite{Asta}. 
By adding viscoelastic dissipation of memory type, Bchatnia and Guesmia \cite{BG} established the so-called general stability. 
More recently, Benaissa and Gaouar \cite{Benaissa} studied strong stability of Lam\'e systems with fractional order boundary damping. With respect to controllability, we refer the reader to, for instance, \cite{Alabau,BL,Lagnese,Lions,Liu}.   

\medskip

Our objective in the present article is different and goes further than considering stabilization. We are concerned with longtime dynamics of Lam\'e systems under nonlinear forces. 
Here, the above linear system (\ref{1}) becomes 
\begin{eqnarray}
\left\lbrace
\begin{array}{ll}
\partial^2_{t} u - \mu \Delta u - (\lambda+\mu) \nabla {\rm div} u + \alpha \partial_t u + f(u) = b & \text{in} \,\; \Omega\times\mathbb{R}^{+}, \smallskip \\
u=0 &  \text{on} \,\; \partial \Omega\times\mathbb{R}^{+}, \smallskip \\
u(0)=u_0, \,\; \partial_t u(0)=u_1 &  \text{in} \; \Omega,
\end{array} 
\right.
\label{problem}
\end{eqnarray}
where $\alpha \partial_t u$ ($\alpha >0$) represents a frictional dissipation, 
$f(u)$ stands for a nonlinear structural forcing, and $b=b(x)$ represents some external force. 
As far as we know, the long-time dynamics of semilinear Lam\'e systems \eqref{problem} has not been studied before. 
We present two main results. 
Firstly, we establish the existence of global attractors with finite fractal-dimensional. 
Secondly, by taking $\lambda + \mu = \varepsilon >0$, we study the upper semicontinuity of 
attractors with respect to $\varepsilon \to 0$. 

\smallskip 

In what follows we summarize the main contributions of the paper.

\smallskip

$(i)$ Our first result establishes existence of global attractors for dynamics of problem (\ref{problem}) under nonlinear forces with critical growth $|f_i(u)| \approx |u|^p + |u_i|^3$, $p<3$, $i=1,2,3$. Under careful energy estimates, we show that the system is gradient and quasi-stable in the sense of \cite{chueshov-book,Von}. Then we conclude that the attractors are smooth and have finite fractal dimension. See Theorem \ref{3.16}.

\smallskip

$(ii)$ In Section \ref{sec-phys}, we discuss the physical meaning of the limit case $\lambda + \mu \to 0$ in real world applications. This arises mainly in Seismology.

\smallskip

$(iii)$ Finally, setting $\varepsilon = \lambda + \mu \to 0$, we consider the $\varepsilon$-problem   
$$
\partial^2_{t} u - \mu \Delta u - \varepsilon  \nabla {\rm div} u + \alpha \partial_t u + f(u) = b,
$$
depending on a parameter $\varepsilon \ge 0$. In Theorem \ref{singular} we show that the weak solutions of $\varepsilon$-problem 
converges to the vectorial wave equation with $\varepsilon=0$. Then we provide all necessary analysis to prove that 
corresponding family of attractors $\mathcal{A}_{\varepsilon}$ is upper semicontinuous with respect to $\varepsilon \to 0$. This is given in a suitable phase space. See Theorem \ref{upper}.  	

\section{Preliminaries}\label{sec-well-posed}
\setcounter{equation}{0}

\subsection{Physical aspects of $\lambda + \mu \to 0$} \label{sec-phys}

From the Hooke law and from the constitutive law (\ref{ten}) referring to elastic bodies, one derives the equation
\begin{equation}\label{ex}
\rho\partial^2_{t} u-\mu \Delta u-(\lambda+\mu)\nabla{\rm div}u=\rho \mathcal{F},
\end{equation}
which may represent the displacement of vector particles for an elastic, isotropic and homogeneous body subject to external forces   $\mathcal{F}$. 

\smallskip 

In Poisson \cite{Po1829}, Timoshenko \cite{Ti1953}, Hudson \cite{Hu1980}, among others, it has been shown that equation \eqref{ex} provides information about different {\it body waves}. 
In a scalar sense  ($P$-waves), where the notation ${\rm div}u$ stands for fractions of volume changes from the strain tensor, it   explains the behavior of compression and rarefaction in the interior of the body. From the mathematical point of view, it can be given by the identity
$$
\partial^2_{t}({\rm div}u)-\alpha^2\Delta ({\rm div}u)={\rm div}\mathcal{F},
$$
where  $\alpha=\sqrt{\frac{\lambda+2\mu}{\rho}}$ represents  speeds of wave propagation.

\smallskip 

On the other hand, by considering the case   $\nabla \times u$, one obtains the behavior of vector waves ($S$-waves) that model  small rotations of lineal elements from shear forces acting within the body. In this way, the following equation arises
$$
\partial^2_{t} (\nabla \times u)-\beta^2\Delta (\nabla \times u)=\nabla \times\mathcal{F},
$$
where $\beta=\sqrt{\frac{\mu}{\rho}}$ means the speeds of $S$-wave propagation.

\smallskip 

The analysis of the dynamics for  \eqref {ex} has shown great applications in the effect of seismic waves on various materials
(e.g. harzburgite, garnet, pyroxenite, amphibolites, granite, gas sands, quartz, etc), where the propagation of the $P$-waves represents the change of volume in the interior of the body under compression and dilatation in the wave direction, see Figure   \ref{fig-w}(b), whereas the $S$-waves are cross displacements that produce vibrations in a perpendicular direction (normal to the traveling wave), see Figure \ref{fig-w}(c).

\begin{figure}[htb] 
	\begin{center}
		\includegraphics[scale=0.45]{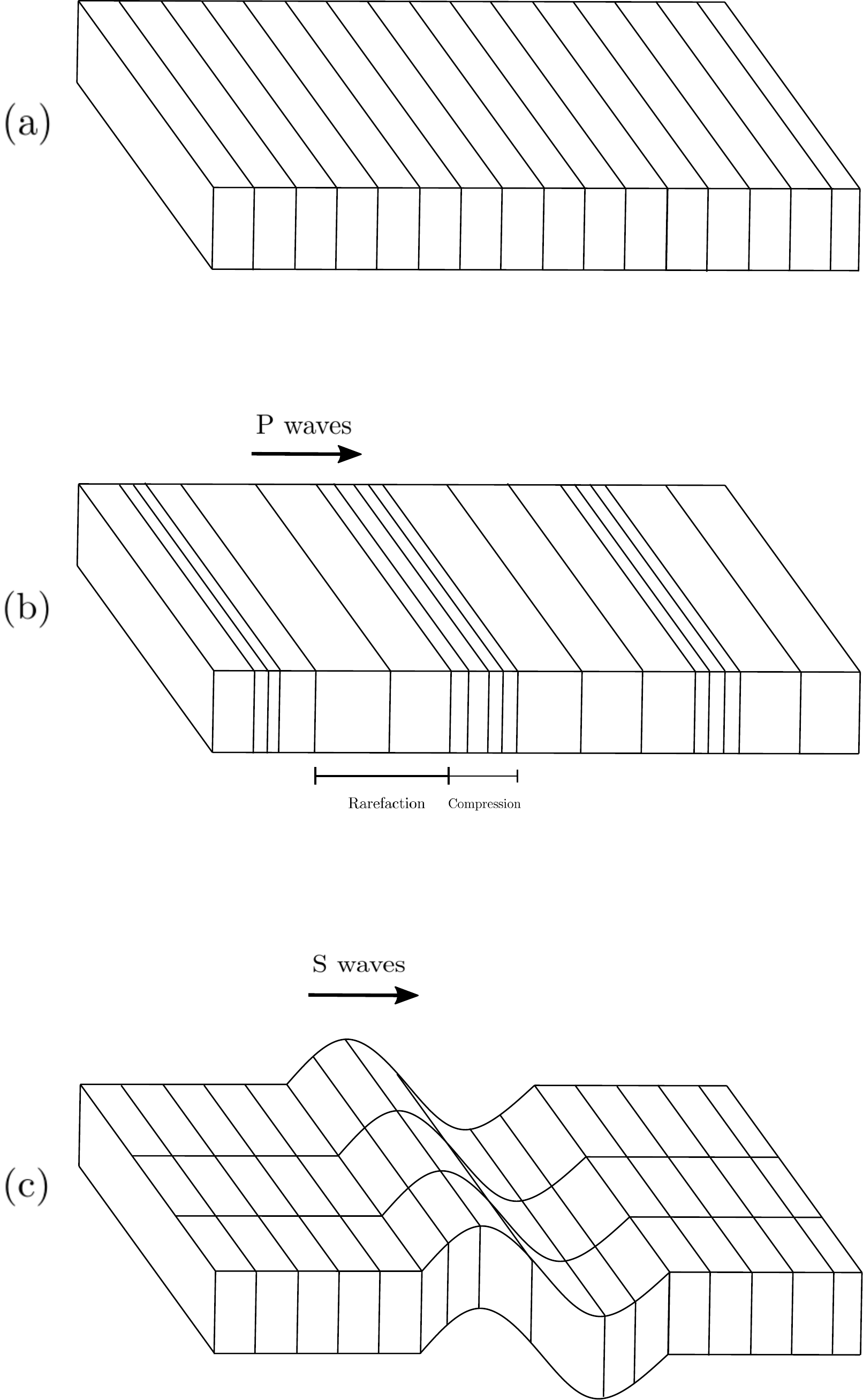} 
	\end{center}
	\caption{\small In (a) we have the elastic body in a rest position.  In (b) we have the effect of $P$-waves propagation on the material, where small contraction and dilation are produced in the same direction of the wave propagation. In (c) we exemplify the effect of transversal $S$-waves on the material, which are generated from the shear forces and are effective in normal directions with respect to the direction of the wave propagation.}
	\label{fig-w}
\end{figure}

A general existing scenario is when earthquakes generate shear waves, say $S$-waves, that are more effective than compression waves, say $P$-waves, and therefore the most damage on the body displacements is  due to the ``stronger'' vibrations caused by $S$-waves.
On the other hand, $P$-waves commonly  
propagate at a higher speed in relation to $S$-waves, reaching   their highest speed, namely, the highest value for $  \beta $, near the basis of the body. Thus, from this viewpoint, it is worth mentioning that the approximation  $\lambda \to -\mu$ symbolizes the approaching of the velocities with respect to $S$-waves in relation to $P$-waves.

\smallskip 

For instance, when one considers the approach of $\lambda$ to $-\mu$  on sedimentary rocks, one has atypical cases  concerning bulk modulus or 
Poisson's ratio. This is the case when one considers e.g.
$\lambda<-\frac{2\mu}{3}$ which is the case where we have negative bulk modulus
or when $\lambda\sim-\mu$ which is the case where the Poisson's ratio is not defined, being $\pm \infty$  in the left or right approximation, respectively. These results seem to contradict the physical notion  that we have  regarding the study of thermodynamics on this type of materials, but several studies show that the compressibility of the material is  closely related to the constant $\lambda$ instead of approximations coming from the    bulk modulus or the Poisson's ratio, see e.g.  Goodway  \cite{AVO}.

\smallskip 

Other examples of such approximations are considered as follows. Indeed, in Moore et al. \cite{Moore} the authors reveal the possibility of considering negative incremental bulk modulus on  open cell foams
on porous media. Also, Lakes and Wojciechowski \cite{LW} show the possibility of taking  negative Poisson's ratio and bulk modulus for the same type of materials, which proves its   structural stability. These are examples that show  us the existence of materials   (e.g., gas sands \cite{AVO} and open cell foams \cite{Moore,LW}) that, under certain circumstances,  allow us to consider the limit situation of $\lambda$ to negative values. Thus, it makes sense to consider for example   $\lambda \to -\mu$.

\smallskip 

Moreover, in Ji et al. \cite{Ji} 
the authors show that for  quartz materials under a confining pressure of $600$ MPa and a temperature around  $650 \, ^0$C, the transmission between High--Low Quartz demonstrates a significant decreasing in the speed of $P$-wave propagation  $\left(\alpha=\sqrt{\frac{\lambda+2\mu}{\rho}}\right)$  in relation to the perturbation of the speed of $S$-wave propagation $\left(\beta=\sqrt{\frac{\mu}{\rho}}\right)$. Therefore, to  consider  the approximation
$$
\frac{\alpha}{\beta}   \to 1 \quad \mbox{ wich means } \quad \lambda\to-\mu 
$$
in the dynamic of seismic waves, it is equivalent to study the state of transition between   High--Low Quartz in materials (say rocks)   containing quartz (as for example granite, diorite, and felsic gneiss)
and its behavior with respect to the wave speeds of propagation for  transverse and compressible waves in the material, under  proper conditions of temperature and pressure.

\subsection{Assumptions}

The following assumptions shall be considered throughout this paper for the functions defined on a bounded domain $\Omega \subset \mathbb{R}^3$ 
 with smooth boundary $\partial\Omega$.

\begin{enumerate}
	\item[(A1)]  The damping coefficient $\alpha$ and the Lam\'e coefficients $\lambda, \mu$ fulfill
\begin{equation} \label{2}
\alpha, \mu>0 \ \ \mbox{ and } \ \ \lambda\in\mathbb{R} \ \ \mbox{with} \ \ \mu +\lambda \geq 0. 
\end{equation}
	
	\item[(A2)] The external vector force $b$ satisfies
	\begin{equation}\label{hip-h}
	b \in (L^2(\Omega))^3.
	\end{equation}

	\item[(A3)]  The nonlinear vector field  $f=(f_1,f_2,f_3)$ is assumed to satisfy: there exist a vector field $g=(g_1,g_2,g_3) \in (C^1(\mathbb{R}^3))^3$, and functions $G \in C^2(\mathbb{R}^3)$ and $h_i \in C^2(\mathbb{R})$, $i=1,2,3$, such that 
	
$$
f_i(u_1,u_2,u_3)=g_i(u_1,u_2,u_3) + h_i(u_i), \quad i=1,2,3,
$$
$$
f_i(0)=g_i(0)=h_i(0)=0, \quad i=1,2,3,
$$
$$
g=(g_1, g_2, g_3)=\nabla G.
$$
In addition, there exist  constants $M, m_f \geq 0$  such that
	\begin{align}
	f(u) \cdot u -G(u) - \sum_{i=1}^3 \int_0^{u_i} h_i(s)ds \geq - M |u|^2 - m_f, \,\,\forall u \in \mathbb{R}^3, \label{9.1}\\
	G(u) +\sum_{i=1}^3 \int_0^{u_i} h_i(s)ds \geq -M |u|^2 -m_f, \,\,\forall u \in \mathbb{R}^3, \label{4.1}
	\end{align}
with 
\begin{equation}\label{4.2}
	0 \leq M < \frac{\mu \lambda_1}{2},
	\end{equation}
where $\lambda_1>0$ denotes the first eigenvalue of the Laplacian operator $-\Delta$. Moreover,
with respect to functions $g_i$ and $h_i$, $i=1,2,3$, we additionally assume: 
	\begin{itemize}
	
	\item $g$ fulfills the subcritical growth restriction: there exist $1 \leq p < 3 $ and $M_g >0$ such that, for $i=1,2,3$,
		\begin{equation}\label{4.3}
		|\nabla g_i(u)| \leq M_g (1+ |u_1|^{p-1}+|u_2|^{p-1}  + |u_3|^{p-1}), \ \ \forall \, u=(u_1,u_2,u_3) \in \mathbb{R}^3.
		\end{equation}
		
		\item For each $i=1,2,3$, $h_i$ fulfills the critical growth restriction: there exists a constant $c_h>0$ such that 
		\begin{align}\label{aa}
		|h'_i(x)| \leq  c_h(1+|x|^2),\quad  \forall \ x \in \mathbb{R}, \ i=1,2,3. 
		\end{align}
	\end{itemize}

\end{enumerate}

\subsection{Functional setting}
We denote the inner product in $L^2(\Omega)$ by $\left\langle u , v\right\rangle = \int_{\Omega} uv dx$ for $u,v \in L^2(\Omega)$. For the sake of simplicity, we use the same notation to the inner product in $(L^2(\Omega))^3$, that is, given $u=(u_1,u_2,u_3), v=(v_1,v_2,v_3) \in (L^2(\Omega))^3$,
\begin{align*}
\left\langle u , v\right\rangle := \sum_{i=1}^3 \left\langle u_i,v_i \right\rangle. 
\end{align*}
Similarly, $\left\langle \nabla \cdot, \nabla \cdot \right\rangle $ stands for the inner product in $H_0^1(\Omega)$ as well as the inner product in $(H_0^1(\Omega))^3$. Thus, given $u=(u_1,u_2,u_3), v=(v_1,v_2,v_3) \in (L^2(\Omega))^3$,
\begin{align*}
\left\langle \nabla u , \nabla v\right\rangle := \sum_{i=1}^3 \left\langle \nabla u_i, \nabla v_i \right\rangle. 
\end{align*}
In addition,  for $p>0$, we denote the norms in the spaces $L^p(\Omega)$ and $(L^p(\Omega))^3$  by $|\cdot|_p$ and $\|\cdot \|_p$, respectively, that is,
\begin{align*}
&|u|_p:=\left(\int_{\Omega} |u|^pdx \right)^{\frac{1}{p}}, \;\; u \in   L^p(\Omega),\\
&\|u\|_p^p:=\sum_{i=1}^3 |u_i|_p^p, \;\; u=(u_1,u_2,u_3) \in   (L^p(\Omega))^3 .
\end{align*}   
In  particular, for   $p=2$, one reads  
$$
\|u\|_2^2=\left\langle u, u\right\rangle \ \mbox{ for } \ u \in   (L^2(\Omega))^3   \ \ \mbox{and} \ \  
|u|_2^2=\left\langle u, u\right\rangle \ \mbox{ for } \ u \in   L^2(\Omega).
$$
The elasticity operator $\mathcal{E}$, with domain $D(\mathcal{E}):= (H^2(\Omega) \cap H_0^1(\Omega))^3 $, is given by  
\begin{equation}\label{operator}
 {\mathcal{E}} u=-\mu \Delta u - (\lambda +\mu)\nabla(\nabla \cdot u).
\end{equation}
We consider the Hilbert space $\left( (H_0^1(\Omega))^3, \left\langle \cdot, \cdot \right\rangle_{e} \right)$, where the inner product $\left\langle \cdot, \cdot \right\rangle_{e} $ is given by 
\[\left\langle v,w \right\rangle_{e} =  \mu \left\langle \nabla v,\nabla w \right\rangle+ (\lambda+\mu) \left\langle \text{div}u, \text{div}w \right\rangle. \]
\begin{remark}\label{lemma0.1} 
	Under the above notations, it is easy to verify that the norms $\| \cdot \|_{e}^2 :=\sqrt{\left\langle \cdot, \cdot \right\rangle_{e}}$ and $\| \nabla \cdot \|_2^2 := \sqrt{\left\langle \nabla \cdot, \nabla \cdot \right\rangle}$ 
	are equivalent in $(H_0^1(\Omega))^3$. More precisely, one has
	\begin{align*}
	\mu\|\nabla u \|_2^2 \leq \| u \|_{e}^2 \leq a_0 \| \nabla u \|_2^2, \ \ \forall \, u=(u_1,u_2,u_3)\in(H_0^1(\Omega))^3,
	\end{align*}	
	where $ a_0=\max\{\mu, 3(\lambda+\mu)\}$. 
\end{remark}

Additionally, if $u \in D(\mathcal{E}) $ and $v \in (H_0^1(\Omega))^3$, then it is easy to verify that 
\begin{align}\label{9.2}
\left\langle \mathcal{E} u  ,v \right\rangle = \left\langle u, v \right\rangle_{e}.
\end{align}
From \eqref{9.2}, Remark \ref{lemma0.1} and the compact embedding of $H_0^1(\Omega)\hookrightarrow L^2(\Omega)$, one sees  that $\mathcal{E}$ is a positive self-adjoint operator. We denote the fractional power associated to $\mathcal{E}$ by $\mathcal{E}^r$ with domain $X^r := D({\mathcal{E}}^r)$, which is  endowed with the natural inner product $\left\langle \cdot, \cdot \right\rangle_{r} := \left\langle {\mathcal{E}}^r \cdot, {\mathcal{E}}^r \cdot \right\rangle$. In particular,
\begin{align*}
X^0&=((L_2(\Omega))^3; \left\langle \cdot,  \cdot \right\rangle),\\
X^{1/2}&=\left( (H_0^1(\Omega))^3; \left\langle {\mathcal{E}}^{1/2} \cdot, {\mathcal{E}}^{1/2} \cdot \right\rangle \right),\\
X^1&=(D({\mathcal{E}}); \left\langle \mathcal{E} \cdot, \mathcal{E} \cdot \right\rangle).
\end{align*} 

\begin{remark}
	From Riesz's Theorem along with density arguments  and continuity, we have
$$
\left\langle u, v \right\rangle_{1/2}=\left\langle u, v \right\rangle_{e}, \quad \forall \ u, v \in (H_0^1(\Omega))^3.
$$	
\end{remark} 

Finally, we define the (Hilbert) weak  phase space $\mathcal{H}:= X^{1/2}\times X^0$ with the usual inner product and induced norm $\|\cdot \|_\mathcal{H}$; and  the (Hilbert) strong  phase space $
\mathcal{H}^1:=X^1 \times X^{1/2} .
$

\subsection{Well-posedness and energy estimates}

Under the above assumptions and notations, we are able to state the Hadamard well-posedness of   $(\ref{problem})$. 
We start by denoting
\begin{equation} \label{eq}
U=\left[\begin{array}{c}
u \\
\partial_t u
\end{array}\right], \ \  \mathbb{E}=\left[\begin{array}{cc}
0 & -I \\
\mathcal{E} & \alpha
\end{array}\right], \ \ \mathbb{F}=\left[\begin{array}{cc}
0 & 0 \\
f(\cdot) & 0
\end{array}\right], \ \ \mathbb{B}=\left[\begin{array}{c}
0 \\
b(x)
\end{array}\right].
\end{equation}
Then, problem (\ref{problem}) is equivalent to the Cauchy problem

\begin{equation}\label{abs-system}
\partial_t U+\mathbb{E}U+\mathbb{F}U=\mathbb{B}, \ \ U(0)=\left[\begin{array}{c}
u_0 \\
u_1
\end{array}\right],
\end{equation}
where $\mathbb{E}:D(\mathbb{E})\subset \mathcal{H}\to \mathcal{H}$ with domain
$$
D(\mathbb{E})=\{ (u,v) \in \mathcal{H} \ | \  \mathcal{E}u+\alpha v \in X^0, \,  v \in X^{1/2}\}=\mathcal{H}^1.
$$

\begin{theorem}[Well-posedness] \label{Theorem3.4} Let us assume that $(\ref{2})$-$(\ref{aa})$ hold. Then,
	\begin{itemize}
		\item[$(i)$] For $(u_0, u_1) \in \mathcal{H}$,  system $(\ref{abs-system})$ possesses a unique mild solution 
		\begin{align*}
		U \in C(\mathbb{R}^+;\mathcal{H}).
		\end{align*} 
		
		\item[$(ii)$] For $(u_0, u_1) \in \mathcal{H}^1$, system $(\ref{abs-system})$ possesses a unique regular solution 
		\begin{align*}
		U \in C(\mathbb{R}^+;\mathcal{H}^1).
		\end{align*}
		
		\item[$(iii)$] For any $T>0$ and any bounded set $B \subset \mathcal{H}$, there exists a constant $C_{BT}>0$ such that for any two solutions $z^i = (u^i,\partial_{t} u^i)$ of \eqref{abs-system} with initial data $z_0^i \in B$, $i=1,2$, we have  
		\begin{align*}\label{cc}
		\|z^1(t)- z^2(t)\|_{\mathcal{H}}^2 \leq C_{BT} \|z_0^1-z_0^2\|_{\mathcal{H}}^2. 
		\end{align*}		
	\end{itemize}
\end{theorem}

\begin{proof}
It is easy to check that operator $\mathbb{E}$ set in \eqref{eq} is a maximal monotone operator and also, under assumption (A3),   $\mathbb{F}$ is a locally Lipschitz on $\mathcal{H}$. Therefore, applying the classical theory of linear semigroups, see e.g. \cite{arrieta,hale,pazy}, items $(i)$-$(ii)$ are concluded. The continuous dependence $(iii)$ is also obtained by using standard computations in the difference of solutions.
\end{proof}


In what follows we give some useful inequalities involving the energy functional. 
The total energy functional associated with problem (\ref{problem}) is given by 
\begin{equation}\label{16}
E(t)=\frac{1}{2}\| (u, \partial_t u)\|_{\mathcal{H}}^2 + \int_{\Omega}G(u)dx + \sum_{i=1}^3 \int_{\Omega} \int_0^{u_i} h_i(s)ds dx - \left\langle  b(x), u\right\rangle.
\end{equation}

\begin{proposition}
Under the hypotheses $(\ref{2})$-$(\ref{aa})$, we have:
\begin{itemize}
\item[$(i)$] the energy $E(t)$ is non-increasing with $E(t)\leq E(0)$ for all $t\geq0;$

\item[$(ii)$] there exist positive constants $K_1, K_2$ and $K_3$ such that 
	\begin{equation}\label{18}
	K_2  \| (u, \partial_t u)\|_{\mathcal{H}}^2  - K_3 \leq E(t)\leq  K_1  \| (u, \partial_t u)\|_{\mathcal{H}}^4 + K_3, \ \ \forall \, t\geq0.
	\end{equation}
\end{itemize}
\end{proposition}
\begin{proof}
$(i)$ Taking the multiplier $u_t$ in problem \eqref{problem}, then a straightforward computation leads us to
	\begin{align} \label{17}
	E'(t) &= 
	-\alpha \|\partial_t u\|_2^2 \leq 0 , \quad  \forall  \ t >0,
	\end{align}
from where it readily follows the stated in item $(i)$. 	

\medskip 

\noindent $(ii)$ From conditions  $(\ref{2})$-$(\ref{aa})$ and Young's inequality with $\epsilon>0$,  the expression
	\begin{align*}
	I =\int_{\Omega}G(u)dx + \sum_{i=1}^3 \int_{\Omega} \int_0^{u_i} h_i(s)ds dx - \left\langle  b(x), u\right\rangle
	\end{align*}
	can be estimated from below and above as follows 
\begin{align*}
I & \geq \, - \ m_f |\Omega|- \frac{\epsilon}{4}\|b\|_2^2 - \left(\frac{M}{\lambda_1\mu} + \frac{1}{\lambda_1\mu \epsilon}\right) \|(u,\partial_t u)\|_{\mathcal{H}}^2,\\
I & \leq C_{f}|\Omega|+  \frac{1}{2}\|b\|_2^2 + \frac{C_{g}}{\mu^{\frac{p+1}{2}}}\| (u,\partial_t u)\|_{\mathcal{H}}^{p+1} \\
	& \ \ \ \ + \frac{C_{h}}{\mu^2}\| (u,\partial_t u)\|_{\mathcal{H}}^{4} +\frac{1}{2\sqrt{\lambda_1}\mu} \|(u, \partial_t u)\|_{\mathcal{H}}^2,
\end{align*} 
	where the positive generic constants depend on their index and some embedding with $H_0^1(\Omega)$, for example $C_{h}$ 
	depends on the constant $c_h$ in (\ref{aa}) and the compact embedding $H_0^1(\Omega) {\hookrightarrow} L^4(\Omega)$. From this and the definition of $E(t)$ in \eqref{16}, we infer 
	\begin{align*}
	E(t)  &\leq C_{f}|\Omega| +\frac{1}{2}\|b\|_2^2 + \frac{1}{2\sqrt{\lambda_1}\mu}+ \frac{C_{g}}{\mu^{\frac{p+1}{2}}}   \\
	& \ \ \  +\left(\frac{1}{2}+\frac{C_{g}}{\mu^{\frac{p+1}{2}}} + \frac{C_{h}}{\mu^2}\right) \|(u,\partial_t u)\|_{\mathcal{H}}^4,\\
	E(t) &\geq -m_f |\Omega|- \frac{\epsilon}{4}\|b\|_2^2 + \left(\frac{1}{2} - \frac{M}{\lambda_1\mu} - \frac{1}{\lambda_1 \mu \epsilon}\right) \|(u,\partial_t u)\|_{\mathcal{H}}^2.
	\end{align*}
	Therefore, from a proper choice of $\epsilon>0$ and using condition (\ref{4.2}), one can conclude the  existence of  positive constants $K_1,K_2$ and $K_3$ satisfying \eqref{18}.
\end{proof}

\begin{remark}\label{remark1}
We emphasizes  that above constants $K_1, K_2$ and $K_3$ in $(\ref{18})$ do not depend on the parameter $\lambda$.
\end{remark}

\section{Long-time dynamics}\label{sec-long-time-dyn}
\setcounter{equation}{0}

From Theorem \ref{Theorem3.4}, one can define a dynamical system  $(\mathcal{H}, S(t))$ associated with problem $(\ref{problem})$, where the  evolution operator  $S(t)$ corresponds to a non-linear $C_0$-semigroup (locally Lipschitz) on $\mathcal{H}$.

 Our main goal in this section is to prove that $(\mathcal{H}, S(t))$ possesses a finite dimensional global attractor $\mathcal{A}$ as well as to reach its qualitative properties such as characterization and regularity. To this end, we first recall some concepts in the theory of dynamical systems, by following e.g. the references \cite{chueshov-book,Von}.

\subsection{Some elements of dynamical systems} 

For the sake of completeness, we recall some basic facts on dynamical systems.

\begin{itemize}
\item A {\it global attractor} for a dynamical system $(\mathcal{H}, S(t))$ is a compact set $\mathcal{A} \subset \mathcal{H}$ which is fully invariant and uniformly attracting, it means, 	for any bounded subset $B \subset \mathcal{H}$
	\[ S(t)\mathcal{A}=\mathcal{A} \text{ and }  \lim_{t\to \infty } d_\mathcal{H} (S(t)B,\mathcal{A}) =0.  \] 
	\item The {\it fractal dimension} of a compact set $B \subset \mathcal{H}$ is defined  as
$$ \dim_{f} B = \limsup_{\epsilon \to 0} \frac{ln N_\epsilon (B)}{ln(1/\epsilon)},$$  
	where $N_\epsilon (B)$ is the minimal number of closed balls of radius $2 \epsilon$ necessary to cover $B$.
	
	\item The set of {\it stationary points} $\mathcal{N}$ of a dynamical system $(\mathcal{H}, S(t))$ is defined as
	\begin{align*}
	\mathcal{N} = \left\{ V \in \mathcal{H} \ |\ S(t)V=V, \ \ \forall \, t >0 \right\}.
	\end{align*}
	
	\item A dynamical system $(\mathcal{H},S(t))$ is called {\it gradient} if there exists a strict Lyapunov functional $\Psi$, that is, for any $z \in \mathcal{H}$, $\Psi(S(t)z)$ is decreasing with respect $t \geq 0$ and $\Psi$ is constant on the set of stationary points $\mathcal{N}$.
	
	\item Given a set $B \subset \mathcal{H}$, its {\it unstable manifold}  $W^u(B)$ is the set of points $z \in \mathcal{H}$ that belongs to some complete trajectory $\{y(t) \}_{t \in \mathbb{R}}$ and satisfies 
	\[  y(0) =z  \text{ and }  \limsup_{t \to -\infty} \text{dist}(y(t),B)=0. \]   
	 	
	\item {\it Quasi-stability.} Let $X,Y$ be reflexive Banach spaces with compact embedding $X \overset{c}{\hookrightarrow} Y$ and $\mathcal{H}=X \times Y$. Let us suppose  $(\mathcal{H},S(t))$ is given by 
	\[  S(t)z =(u(t), \partial_t u(t)) , \,\, z=(u_0,u_1) \in \mathcal{H}, \]
	where 
	\[ u \in C(\mathbb{R}^{+} ; X ) \cap C^1 (\mathbb{R}^{+} ; Y),    \]
	Then, $(\mathcal{H},S(t))$ is called {\it quasi-stable} on a set $B \subset \mathcal{H}$ if there exists a compact semi-norm $\eta_X$ on $X$ and non-negative scalar functions $a_1(t)$ and $a_3(t)$ locally bounded in $\mathbb{R}^{+}$ and $a_2(t) \in L^1(\mathbb{R}^{+})$ with $\lim_{t \to \infty} a_2(t)=0$ such that
	\begin{align*}
	\|S(t)z^1 -S(t)z^2\|_\mathcal{H}^2 \leq a_1(t)\|z^1 -z^2\|_\mathcal{H}^2,
	\end{align*}    
	and
	\begin{align*}
	\|S(t)z^1 -S(t)z^2\|_\mathcal{H}^2 \leq a_2(t)\|z^1 -z^2\|_\mathcal{H}^2 + a_3(t)\sup_{0 \leq s \leq t} \left[ \eta_X (u^1(s)  -u^2(s) ) \right]^2, 
	\end{align*}
	for any $z^1, z^2 \in B$. 
\end{itemize}

\begin{proposition}[{\cite[Corollary 7.5.7]{Von}}]\label{corollary}
Let $(\mathcal{H},S(t))$ be a gradient asymptotically smooth dynamical system. Additionally, if its Lyapunov function $\Psi(x)$ is bounded from above on any
	bounded subset of $\mathcal{H}$,  the set $\Psi_R = \left\{x \in \mathcal{H} : \Psi(x) \leq R\right\}$ is bounded for every $R$ and  the set $\mathcal{N}$ of stationary points of $(\mathcal{H},S(t))$ is bounded, then $(\mathcal{H},S(t))$ possesses a compact global attractor characterized by $ \mathcal{A}=W^u(\mathcal{N})$.
\end{proposition}

\begin{proposition}[{\cite[Proposition 7.9.4]{Von}}]\label{corollary1}
Let us assume that the dynamical system $(\mathcal{H},S(t))$  is quasi-stable on every bounded forward invariant set $B\subset \mathcal{H}.$ Then,   $(\mathcal{H},S(t))$
	is asymptotically smooth.
\end{proposition}

\begin{proposition}[{\cite[Theorem 7.9.6]{Von}}]\label{corollary2}
	Let  $(\mathcal{H},S(t))$ a quasi-stable dynamical system. If  $(\mathcal{H},S(t))$ possesses a 
compact global attractor $\mathcal{A}$ and is quasi-stable on $\mathcal{A}$, hen the attractor $\mathcal{A}$  has a finite fractal dimension $\dim_{f} \mathcal{A}<\infty.$

\end{proposition}

\subsection{Main result and proofs}

We are now in condition to state and prove the main result concerning global attractors associated with problem $(\ref{problem})$. It reads as follows.

\begin{theorem}\label{3.16}
Under the assumptions $(\ref{2})$-$(\ref{aa})$, we have:	
\begin{itemize}
	\item[$(i)$] The dynamical system $(\mathcal{H}, S(t))$ corresponding to problem \eqref{problem} has a unique   global attractor $\mathcal{A}$ with finite fractal dimension $\dim_{f} \mathcal{A}<\infty$, and is  characterized by the unstable manifold $ \mathcal{A}=W^u(\mathcal{N})$ emanating from the set of stationary points $\mathcal{N}$ of $(\mathcal{H}, S(t))$.
	
		\item[$(ii)$]  Moreover, if  $h_i=0$, $i=1,2,3$, then $\mathcal{A}$ is bounded in the strong phase space $\mathcal{H}^1$. In particular, any full trajectory $\{(u(t),\partial_t u(t)), t \in \mathbb{R}\}$ that belongs to $\mathcal{A}$ has the following regularity properties
		\begin{equation}\label{reg-1}
		\partial_t u \in L^{\infty}(\mathbb{R};(H_0^1(\Omega))^3) \cap C(\mathbb{R};(L^2(\Omega))^3),\,\,\, \partial^2_{t} u \in L^{\infty}(\mathbb{R};(L^2(\Omega))^3),
		\end{equation}
		and there exists $R>0$ such that 
		\begin{equation}\label{reg-2}
		\|(\partial_{t} u(t),\partial^2_{t} u(t))\|_\mathcal{H}^2 \leq R^2,
		\end{equation}
		where $R$ does not depend on $\lambda$.
\end{itemize}	
\end{theorem}

The proof of Theorem \ref{3.16} will be concluded at the end of this section as a consequence of some technical results provided in the sequel.

\subsubsection{Gradient property}

\begin{lemma}\label{Lemma3.9}
Under the assumptions of Theorem $ \ref{3.16}, $ let us define the functional 
	\begin{equation*}
	\begin{array}{rcl}
	\Psi:  \, \mathcal{H} & \rightarrow & \mathbb{R}  \\
    z & \mapsto  & \Psi(z):=\Psi(u, v)    
	\end{array}
	\end{equation*}
given by	
\begin{equation}
\Psi(u, v)=\frac{1}{2}\| (u, v)\|_{\mathcal{H}}^2 + \int_{\Omega}G(u)dx + \sum_{i=1}^3 \int_{\Omega} \int_0^{u_i} h_i(s)ds dx - \left\langle  b(x), u\right\rangle.
\end{equation}	
Then:
\begin{enumerate}
	\item $\Psi$ is a strict Lyapunov functional;
	
	\item $\Psi(z) \rightarrow \infty $ if and only if $\|z\|_{\mathcal{H}} \rightarrow \infty$;
	
	\item  $\mathcal{N}$  is bounded on $\mathcal{H}$
\end{enumerate}

As a consequence, the dynamical system $(\mathcal{H}, S(t))$ associated with  problem \eqref{problem} is a gradient system.
\end{lemma}

\begin{proof} Let fix $z_0 \in \mathcal{H}$ and recall that $\mathcal{N}$  is the set of stationary points of $(\mathcal{H}, S(t))$. Also, from \eqref{16} one sees that $\Psi(u(t),\partial_{t}u(t))=E(u(t),\partial_{t}u(t)):=E(t)$.  Then, we infer: 
	
\begin{itemize}
		\item From (\ref{17}), it is clear that $\Psi(S(t)z_0)$ is decreasing with respect to time and from (\ref{18}), $\Psi(z)=\Psi(S(0)z) \rightarrow \infty $ if and only if $||z||_{\mathcal{H}} \rightarrow \infty$.
		
\smallskip 
		
		\item 	 Let us consider the stationary problem:		
		\begin{eqnarray}\label{bb}
		\left\lbrace
		\begin{array}{lcr}
		\mathcal{E} u+ f(u) = b(x) & \text{in}& \Omega,\\
		u=0 & \text{on}&\partial \Omega.
		\end{array} 
		\right. 
			\end{eqnarray}
Thus, a simple computation shows that 	$\mathcal{N}$ is given by 		
$$
\mathcal{N}=\left\{(u,0) \in \mathcal{H} \ |\ u \text{ is the solution of } (\ref{bb}) \right\}.
$$ 
In addition, from (\ref{17}) it is easy to prove that  $\Psi$ is constant on $\mathcal{N}$. Finally,  multiplying (\ref{bb}) by $u$, integrating on $\Omega$ and using (\ref{9.1}) and (\ref{4.1}),  we obtain that for any $\epsilon >0$
		\begin{align}\label{eq23}
		\left( 1-\frac{2M}{\lambda_1\mu} -\frac{1}{4 \lambda_1 \mu \epsilon}\right) \|u\|_{e}^2 &\leq 2m_f|\Omega|+ \epsilon \|b\|_2^2, 
		\end{align}
	from where (along with (\ref{4.2})) we conclude that $\mathcal{N}$ is bounded on $\mathcal{H}$, for $\epsilon>0$ properly chosen.
	\end{itemize}
Therefore,  the items 1 - 3 are proved.
\end{proof}
%

\subsubsection{Quasi-stability property}


\begin{proposition}[Stabilizability Estimate]\label{Theorem8}
	Under the assumptions of Theorem $ \ref{3.16},$ let us consider  a bounded subset $B\subset \mathcal{H}$ and  two weak solutions   $\tilde{z}=(v,\partial_t v)$ and $ z=(u, \partial_t u)$ of problem  $(\ref{problem})$  with initial data $\tilde{z}(0)=(v_0,v_1)$, $z(0)=(u_0,u_1)\in B$. Then, 
	\begin{equation}\label{est-est}
	\|\tilde{z}(t)-z(t)\|_{\mathcal{H}}^2 \leq a_2(t)\|\tilde{z}(0)-z(0)\|_{\mathcal{H}}^2+ c(t)\sup_{0\leq s\leq t} \|v-u\|_{p_0}^2,
	\end{equation}
	where $p_0=\max\{4, \frac{6}{4-p}\}< 6$, $b\in L^1(\mathbb{R}^+)$ with $\displaystyle\lim_{t \to \infty} a_2(t)=0$ and $c(t)$ is a locally bounded function. 
\end{proposition}

\begin{proof}
The estimate \eqref{est-est} is one of the main cores of the present article. Its proof   is quite technical and long, and for this reason  we are going to proceed in several steps as follows.

\smallskip 
\noindent{\it Step 1. Setting the difference problem and functionals.} 	Let us denote $w=u-v$. Then, a simple computation shows that $w$ is a solution (in the weak and strong sense) of the following problem
\begin{eqnarray} \label{29}
\left\lbrace
\begin{array}{lcr}
\partial^2_{t} w+ \mathcal{E} w + \alpha \partial_t w +f(u)-f(v) =0 &\text{on}& \Omega\times\mathbb{R}^{+}, \smallskip \\
w=0 &\text{in}& \partial \Omega\times\mathbb{R}^{+}, \smallskip  \\
w(x,0)=u_0(x)-v_0(x), \ x\in \Omega, && \smallskip  \\
\partial_t w(x,0)=u_1(x)-v_1(x) , \ x\in\Omega.&&
\end{array} 
\right.
\end{eqnarray}
The energy associated with system  (\ref{29}) is given by 
\begin{equation}\label{Xi}
\Xi(t):= \frac{1}{2}\| (w, \partial_t w)\|_{\mathcal{H}}^2= \frac{1}{2} \|\tilde{z}(t)-z(t)\|_{\mathcal{H}}^2, \ \ t\geq0.
\end{equation}
We also set the functional
\begin{align*}
\chi(t) &= \left\langle w, \partial_{t} w \right\rangle,
\end{align*}
and the perturbed energy functional
\begin{align*}
\Upsilon(t) = \epsilon_1 \Xi(t) + \epsilon_2\chi(t) ,
\end{align*}
where the constants $\epsilon_1, \epsilon_2 >0$ will be chosen later.

\smallskip 
\noindent{\it Step 2. Equivalence.} 
	There exist constants $C_1, C_2 >0$ such that 
\begin{equation}\label{26}
C_2\Xi(t) \leq \Upsilon(t) \leq C_1 \Xi(t).
\end{equation}		
Indeed, the inequalities in \eqref{26} follow by taking $K'=\max\{\frac{C^2}{\mu},1\}$, 
$\epsilon_1 > \epsilon_2K'  $, $C_2=\epsilon_1 - \epsilon_2K'$ and $C_1=\epsilon_1 + \epsilon_2K' $.	
	
\smallskip 

\noindent{\it Step 3.  Estimate for $\Xi'$.} 	Given $\xi>0$, there exists a constant $C(\xi,B)>0$, which depends on $\xi$ and $B$, such that
\begin{align}\label{Xi'}
\Xi'(t) \leq -{\alpha} \|\partial_t w\|_2^2+C(\xi,B)\|w\|_{\frac{6}{4-p}}^2+\xi\|\partial_t w\|_2^2 + I,
\end{align}
where we set 
\begin{equation}\label{I}
I := \displaystyle\sum_{i=1}^3 \left\langle h_i(v_i)-h_i(u_i), \partial_t w_i \right\rangle.
\end{equation}
 In fact, we first observe that deriving    $\Xi(t)$ and using (\ref{29}), we get
\begin{align*}
\Xi'(t) =&  -\alpha \|\partial_t w\|_2^2 -\left\langle g(u)-g(v), \partial_t w \right\rangle+I .
\end{align*}
Since
\begin{align*}
|\left\langle g(u)-g(v), \partial_t w \right\rangle| \leq \sum_{i=1}^{3} \int_{\Omega} M_g \left\lbrace 1 + \sum_{i=1}^3 |v_i|^{p-1} +  \sum_{i=1}^3  |u_i|^{p-1}   \right\rbrace|w| |\partial_t w_i|dx,
\end{align*}
then applying H\"older's inequality, we obtain  
\begin{align}
|\left\langle g(u)-g(v), \partial_t w \right\rangle| \leq \sum_{i=1}^{3}  \tilde{C}_f\|w\|_{\frac{6}{4-p}} \|\partial_t w_i\|_2 , 
\end{align}  
where  
\begin{align*}
\tilde{C}_f = M_f \left\lbrace |\Omega|^{\frac{p-1}{6}} + \sum_{i=1}^3\|v_i\|_{6}^{p-1} + \sum_{i=1}^3 \|u_i\|_6^{p-1} \right  \rbrace \leq C(B)<\infty,
\end{align*}
 is a constant depending on $B$. Therefore, the estimate \eqref{Xi'} follows from   Young's inequality with $\xi>0$.

\smallskip 
\noindent{\it Step 4.  Estimate for $\chi'$.}  	There exists a constant $C(B)>0$ depending on $B$ such that
\begin{align}\label{chi'}
\chi'(t) \leq -\Xi(t) -\frac{1}{2}\|w\|_e^2 +\frac{\alpha}{2}\| w\|_2^2  + C(B)\|w\|_{4}^2 + \frac{3+\alpha}{2}\|\partial_t w\|_2^2.
\end{align}
Indeed, multiplying (\ref{29})$_1$ by $w$ and integrating on $\Omega$, we obtain
\begin{align*}
\chi' (t) 
&=-\Xi(t) -\frac{1}{2}\|w\|_e^2 + \frac{\alpha}{2}\| w\|_2^2 + \frac{\alpha+3}{2}\|\partial_t w \|_2^2  \\
&\hspace{3cm}- \left\langle g(u)-g(v),w\right\rangle  +\sum_{i=1}^3 \left\langle h_i(u_i)-h_i(v_i) ,w_i \right\rangle.
\end{align*} 
Now, noting that
\begin{align*}
\left\langle g(v)-g(u),w \right\rangle 
 \leq 3 \left\{|\Omega|^{p-1} + \sum_{i=1}^3 \|u_i\|_{p+1}^{p-1} + \sum_{i=1}^3 \|v_i\|_{p+1}^{p-1}  \right\} \|w\|_{p+1}^2  \leq \tilde{C}_B \|w\|_{p+1}^2,
\end{align*}
and 
\begin{align*}
\sum_{i=1}^3 \left\langle h_i(u_i)-h_i(v_i) ,w_i \right\rangle 
 \leq \sum_{i=1}^3( |\Omega|^2 +\|v_i\|_4^2 + \|u_i\|_4^2)|w_i|_4^2 \leq C_B \|w\|_4^2,
\end{align*}  
where the constants $\tilde{C}_B,C_B>0$	depend only on $B$, then the estimate \eqref{chi'} follows.

\smallskip 
\noindent{\it Step 5.  Estimate for $\Upsilon$.} 
There exists  a  constant   $C_3>0$ depending on $B$ such that
\begin{align}\label{23}
\Upsilon(t)\leq e^{-\frac{\epsilon_2 t}{C_1}}\Upsilon(0) +C_3  \int_0^t e^{-\frac{\epsilon_2 }{C_1}(t-s)}\|w(s)\|_{p_0}^2ds + \epsilon_1 e^{-\frac{\epsilon_2t }{C_1}}J , 
\end{align}
where $C_1>0$ comes from \eqref{26} and we set 
\begin{align}\label{Jota}
J:=\int_0^t e^{\frac{\epsilon_2s }{C_1}}Ids =\sum_{i=1}^3 \int_0^t e^{\frac{\epsilon_2s }{C_1}}  \left\langle h_i(v_i(x,s))-h_i(u_i(x,s)), \partial_t w_i(x,s) \right\rangle ds.
\end{align}
First, we note that from \eqref{Xi'} and \eqref{chi'}, one has  
	\begin{align*}
\Upsilon'(t) 
\leq & -\epsilon_2\Xi(t)+\frac{\alpha\epsilon_2}{2}\|w\|_2^2 + \epsilon_2C(B) \|w\|_{4}^2 +\epsilon_1C(\xi,B)\|w\|_{\frac{6}{4-p}}^2   \\
&  \ + \epsilon_1 I + \left( \frac{3\epsilon_2 + \alpha \epsilon_2}{2} +\epsilon_1\xi-\alpha \epsilon_1\right) \|\partial_t w\|_2^2.
\end{align*}
We now choose $\epsilon_1,\epsilon_2,\xi>0$ small enough such that 
\begin{align*}
\epsilon_2K'   < \epsilon_1 \quad \mbox{ and } \quad 
\frac{3\epsilon_2 + \alpha\epsilon_2}{2} +\epsilon_1\xi <\alpha \epsilon_1.
\end{align*}
It is worth mentioning that   $\epsilon_1, \epsilon_2, \xi >0$ do not depend on $\lambda$. Thus, from this choice, setting  $p_0=\max\{\frac{6}{4-p},4\}$ and using  \eqref{26}, there exists a    constant  $C_3=C(B)>0$ such that
\begin{align*}
\Upsilon'(t) & \leq -\frac{\epsilon_2}{C_1}\Upsilon(t) + C_3 \|w\|_{p_0}^2 + \epsilon_1 I,
\end{align*}
from where it follows the estimate \eqref{23} with $J$ given in \eqref{Jota}.  

\begin{remark}
Since 	the choices for	$\epsilon_1, \epsilon_2$ do not depend on   $\lambda$, then  $C_3>0$ is a constant  that does not depend on $\lambda$ as well.	
\end{remark}

	\smallskip 
\noindent{\it Step 6. Estimate for $J$.} 
There exist constants $\gamma_0>0$ and  $C_4>0$	depending on $B$ such that 
\begin{equation}\label{est-J}
J \leq   C_4 e^{\gamma_0 t} \sup_{0<s<t} \|w\|_4^2 + C_4 \int_0^t (\|\partial_t u(s)\|_2 + \|\partial_t v (s)\|_2) e^{\gamma_0 s}  \Upsilon(s)ds.
\end{equation}
Firstly, in view of assumption \eqref{aa}	
and following verbatim the same arguments as in \cite[Lemma 4.9]{Ma}, for any constant $\gamma>0$ and each $i=1,2,3,$  
there exists a constant $K_i'>0$ such that
\begin{align}
\int_0^t e^{\gamma s} &\left\langle h_i(v_i(s))-h_i(u_i(s)), \partial_t w_i(s) \right\rangle ds  \notag\\
& \leq K_i' e^{\gamma t} \sup_{0<s<t} \|w_i\|_4^2 + K_i' \int_0^t \left(| u_i'(s) |_2 + | v_i'(s) |_2\right) e^{\gamma s} |\nabla w_i(s)|_2^2ds.\label{22}
\end{align}	
Therefore, from   \eqref{Jota} and \eqref{22} it prompt follows
	\begin{align*}
J \leq \sum_{i=1}^3K'_i \sup_{0<s<t} \|w\|_4^2 + \max\{K'_i\}  \int_0^t e^{\gamma_0 s} (\| \partial_t u \|_2 + \|\partial_t v \|_2)  \|\nabla w(s)\|_2^2ds,
\end{align*}	
for 	$\gamma_0=\frac{\epsilon_2}{C_1}>0.$
	Additionally, taking $C_4= \max\{K'_1+K'_2+K'_3, \frac{2\max\{K'_i\}}{\mu C_2}\}>0$ and noting that 
	\begin{align*}
	\|\nabla w(s)\|_2^2 \leq \frac{1}{\mu}\| w\|_{e}^2 \leq \frac{2}{\mu} \Xi(s) \leq \frac{2}{\mu C_2} \Upsilon(s),
	\end{align*}
then estimate \eqref{est-J} follows as desired.

\begin{remark}
We emphasize that constants $\gamma_0$ and $C_4$ do not depend on $\lambda$.	
\end{remark}
	
	\smallskip 
\noindent{\it Step 7. Conclusion of the proof.} 	We are finally in position to complete the proof of \eqref{est-est}. Indeed, from \eqref{23} and \eqref{est-J}, there exists a constant $C_5>0$ depending on $B,$ but independently of  $\lambda$, such that 
	\begin{align*}
e^{\gamma_0 t}\Upsilon(t) \leq & \, 	C_5\Upsilon(0) +C_5 e^{\gamma_0 t} \sup_{0<s<t} \|w\|_{p_0}^2  \\
& \ +C_5  \int_0^t (\|\partial_t u(s)\|_2 + \|\partial_t v (s)\|_2) e^{\gamma_0 s}  \Upsilon(s)ds,
\end{align*} 	
and applying Gronwall's inequality,  one gets
\begin{align}
\Upsilon(t)\leq   	C_5\left\lbrace e^{-\gamma_0 t} \Upsilon(0) +  \sup_{0<s<t} \|w\|_{p_0}^2 \right\rbrace 
e^{\left(C_5  e^{-\gamma_0 t}\int_0^t (\|\partial_t u(s)\|_2 + \|\partial_t v (s)\|_2) e^{\gamma_0 s}ds\right)}. \label{31}
\end{align}
Now, from \eqref{18} and (\ref{17}), and also  in view of Remark \ref{remark1}, we  have	
\begin{equation*}
\int_0^t \|\partial_t u(s)\|_2^2ds = -\frac{1}{\alpha}\int_0^t E'(s)ds  \leq \frac{2|E(0)|}{ \alpha} \leq  Q,
\end{equation*}
where $Q >0$ is a constant depending on $B$ and $f$, but independent of $\lambda.$ The same computation holds true for  $\int_0^t \|\partial_t v(s)\|_2^2ds$. Thus,  using H\"{o}lder and Young's inequalities, we obtain
	\begin{align*}
e^{-\gamma_0 t}\int_0^t (\|\partial_t u(s)\|_2 + \|\partial_t v (s)\|_2) e^{\gamma_0 s}ds \leq 2\sqrt{Q}\sqrt{t}\leq \epsilon t + \frac{2Q}{\epsilon},
\end{align*}
for any $t>0$ and $\epsilon>0$. Replacing the latter estimate in  (\ref{31}), we arrive at
	\begin{align*}
\Upsilon(t) &\leq 	C_5e^{(\epsilon C_5 t + \frac{2C_5Q}{\epsilon})}\left\lbrace e^{-\gamma_0 t}\Upsilon(0) + \sup_{0<s<t} \|w\|_{p_0}^2 \right\rbrace .
\end{align*}
	Taking $\epsilon = \frac{\gamma_0}{2C_4}$ and using  (\ref{26}), we have
\begin{equation}\label{last}
\Xi(t) \leq \frac{C_1C_4e^{\frac{Q}{\gamma_0}}}{C_2} e^{\frac{-\gamma_0}{2} t}\Xi(0) + \frac{C_4e^{\frac{Q}{\gamma_0}}}{C_2} e^{\frac{\gamma_0}{2}t} \sup_{0<s<t} \|w\|_{p_0}^2.
\end{equation}
Finally, regarding the definition of $\Xi(t)$, $t\geq0$, in \eqref{Xi} and setting 
\begin{equation}\label{functions}
a_2(t):=\frac{C_1C_4e^{\frac{Q}{\gamma_0}}}{C_2} e^{\frac{-\gamma_0}{2} t} \ \ \mbox{ and  } \ \ a_3(t):=\frac{2C_4e^{\frac{Q}{\gamma_0}}}{C_2} e^{\frac{\gamma_0}{2}t},
\end{equation}
then \eqref{last} leads to \eqref{est-est} as desired. 

The proof of Proposition \ref{Theorem8} is therefore concluded. 
\end{proof}

\begin{corollary}[Quasi-stability]\label{cor-qs}
		Under the assumptions of Theorem $ \ref{3.16},$ the 	dynamical system $(\mathcal{H}, S(t))$ associated with problem \eqref{problem} is quasi-stable on any bounded set $B\subset\mathcal{H}$.	
\end{corollary}
\begin{proof}
	It is a direct consequence of Theorem \ref{Theorem3.4} - $(iii)$ and Proposition \ref{Theorem8} by noting the semi-norm given by $n_{H^{1}_{0}}(v-u)=\|v-u\|_{p_0}$ is compact.
\end{proof}

\subsubsection{Conclusion of the proof of Theorem $ \ref{3.16} $}  

\begin{itemize}
	\item[$(i)$] From Proposition \ref{corollary1} and Corollary \ref{cor-qs}, the 		dynamical system $(\mathcal{H}, S(t))$ related to problem \eqref{problem} is asymptotically smooth. Therefore, using Lemma \ref{Lemma3.9} and Propositions \ref{corollary} and \ref{corollary2}, the conclusion of Theorem $ \ref{3.16} $ - $ (i)$ is complete. 
	
		\item[$(ii)$] In  case $h_i=0$, $i=1,2,3$, then going back to \eqref{I}, one sees that $I=0$ and, consequently, from \eqref{Jota} one gets $J=0$. Thus, \eqref{23} reduces to
		\begin{align*}
	\Upsilon(t)\leq e^{-\frac{\epsilon_2 t}{C_1}}\Upsilon(0) +\frac{C_3 C_1}{\epsilon_2} \sup_{0 \leq s \leq t} \|w(s)\|_{p_0'}^2 \big(1-e^{-\frac{\epsilon_2 }{C_1}t}\big)  ,
	\end{align*}	
$p_0'=\max\{\frac{6}{4-p},p+1\}$.		
In this way, one reaches 	\eqref{last} (respec.  \eqref{est-est}) with 
$$
a_3(t):= \frac{C_3 C_1}{\epsilon_2}   \big(1-e^{-\frac{\epsilon_2 }{C_1}t}\big),
$$	
instead of $a_3(t)$ given in \eqref{functions}.	Thus, $c_{\infty}=\sup _{t \in \mathbb{R}^{+}} a_3(t)<\infty$, and from \cite[Theorem 7.9.8]{Von}, the regularity properties \eqref{reg-1}-\eqref{reg-2} are ensured, that is, the conclusion of Theorem $ \ref{3.16} $ - $ (ii)$ is complete.

Therefore, the proof of Theorem $ \ref{3.16} $ is ended. 
\end{itemize}

\section{Upper semicontinuity}\label{sec-upper-sem}
\setcounter{equation}{0}

Along this section $\varepsilon$ denotes a real number in $[0, 1]$ and assume $\lambda + \mu = \varepsilon$. 
Thus, problem (\ref{problem}) can be rewritten as follows
\begin{eqnarray}
\left\lbrace
\begin{array}{ll}
\partial^2_{t} u - \mu \Delta u - \varepsilon  \nabla {\rm div} u + \alpha \partial_t u + f(u) = b & \text{in} \,\; \Omega\times\mathbb{R}^{+}, \smallskip \\
u=0 &  \text{on} \,\; \partial \Omega\times\mathbb{R}^{+}, \smallskip \\
u(0)=u_0, \,\; \partial_t u(0)=u_1 &  \text{in} \; \Omega,
\end{array} 
\right.
\label{problem-sing}
\end{eqnarray}
In this way, instead of operator \eqref{operator}, we write the $\varepsilon$-operator 
\begin{align*} 
\mathcal{E}_{\varepsilon}u:= -\mu \Delta u - \varepsilon \nabla \text{div}(u), \ \ \mbox{ for } \ \ u=(u_1,u_2,u_3).
\end{align*}
Hereafter, we denote by $P_\varepsilon$ the $\varepsilon$-problem \eqref{problem-sing}  and, in view of  Theorem \ref{3.16}, we also denote by $\mathcal{A}_\varepsilon$ the compact finite dimensional global attractor of its associated dynamical system.
The energy corresponding to $P_\varepsilon$ is still given by (\ref{16}) and denoted here as $E_\varepsilon(t)$.

\smallskip 

Using the same notation as in Section \ref{sec-well-posed}, we define the inner-product 
\[
\left\langle v,w \right\rangle_{\varepsilon} =  \mu \left\langle \nabla v,\nabla w \right\rangle+ \varepsilon \left\langle \text{div}u, \text{div}w \right\rangle. 
\] 
Then, the norm $\|\cdot \|_\varepsilon = \sqrt{\left\langle v,w \right\rangle_{\varepsilon}}$ satisfies that 
\begin{align}\label{up}
\mu\|\nabla \cdot \|_2^2 \leq \| \cdot \|_{\varepsilon}^2 \leq \max\{\mu, 3\} \| \nabla \cdot \|_2^2.
\end{align}
Additionally, let us denote by 
$$
\mathcal{H}_\varepsilon = ((H_0^1(\Omega))^3, \|\cdot\|_\varepsilon) \times ((L^2(\Omega))^3, \|\cdot \|_2)
$$
the space of weak solutions associated to $P_\varepsilon$, and 
$$
\mathcal{H}_0^1 = (D(-\Delta), \|\mu \Delta \cdot\|_2) \times ((H_0^1(\Omega))^3, \|\mu \nabla \cdot\|_2)
$$ the space of strong solutions associated to $P_0$.  

Analogously, we denote by 
$(\mathcal{H}_\varepsilon,S_\varepsilon(t))$ the dynamical system associated with $P_\varepsilon$, and by 
$\mathcal{N}_\varepsilon$, its corresponding set of stationary solutions. The existence of a global attractor $\mathcal{A}_\varepsilon$ 
as well as its properties are ensured by Theorem \ref{3.16}.

\smallskip 

In this section, our main goal is to study the upper semicontinuity of attractors $\mathcal{A}_\varepsilon$ 
with respect to the parameter $\varepsilon \to 0$. More precisely, our main results are presented in Theorems \ref{singular} and \ref{upper}. To this end, we need to prove the following two properties: the existence of an absorbing set which does not depends on $\varepsilon$ and the convergence in some sense of the solutions of $P_\varepsilon$ when $\varepsilon \to 0$. 

\smallskip

For the existence of an absorbing set we need the following result which is a direct consequence of \cite[Remark 7.5.8]{Von}.

\begin{theorem}\label{Remark3.11}
	Under the conditions $(\ref{2})$-$(\ref{aa})$, 
	the following inequality holds true for 
 the attractor $\mathcal{A}$ in Theorem  $ \ref{3.16} $ and $\Psi$ given in Lemma $ \ref{Lemma3.9} $:
	\[  \sup\{\Psi(u,\partial_{t} u): (u,\partial_{t} u) \in \mathcal{A} \}  \leq  \sup\{\Psi(u,0): (u,0) \in \mathcal{N} \}.   \]	
\end{theorem}

With respect to the convergence of solutions, it is important to note that the phase space $\mathcal{H}_\varepsilon$ changes when $\varepsilon \to 0$. So, the convergence of the solutions of $P_\varepsilon$ is \textit{singular} in the same sense proposed in \cite{singular}.

\begin{lemma}\label{Lemma4.1} Under the conditions $(\ref{2})$-$(\ref{aa})$, there exists a bounded absorbing set  $\mathcal{B}$ for $(\mathcal{H}_\varepsilon, S_\varepsilon(t))$, that does not depend on $\varepsilon$.
\end{lemma}

\begin{proof}
	Denoting by $\Psi_\varepsilon$ the Lyapunov functional defined on $\mathcal{H}_\varepsilon$, then from (\ref{18}), Remarks \ref{remark1}  and  Theorem \ref{Remark3.11},
	\begin{align*}
	\sup_{z \in \mathcal{A}_\varepsilon} \|z\|_{\mathcal{H}_\varepsilon}^2  &\leq \frac{\sup_{z \in \mathcal{A}_\varepsilon}\Psi_\varepsilon(z) + K_3}{K_2}\\
	&\leq \frac{\sup_{z \in \mathcal{N}_\varepsilon}\Psi_{\varepsilon}(z) + K_3}{K_2}\\
	&\leq \frac{ K_1\sup_{z \in \mathcal{N}_\varepsilon}\|z\|_{\mathcal{H}_\varepsilon}^4 + 2K_3}{K_2}.
	\end{align*}  
	Thus, from (\ref{eq23}) there exists a constant $R_1$ which does not depend on $\varepsilon$ such that 
	\begin{align*}
	\sup_{z \in \mathcal{A}_\varepsilon} \|z\|_{\mathcal{H}_\varepsilon}^2  \leq R_1^2, \,\,\ \forall \varepsilon \in [0,1] .
	\end{align*}
	Let us define $\mathcal{B}= \left\{z \in \mathcal{H}_0 : \|z\|_{\mathcal{H}_0}^2 \leq R_1+1 \right\}$, then from (\ref{up})  $A_\varepsilon \subset \mathcal{B}$, for all $\varepsilon$.
\end{proof}

\begin{lemma}\label{lemma15}
	Let $B$ be a bounded subset in $\mathcal{H}_0$ 
	and $\{z_\varepsilon =(u_\varepsilon, v_\varepsilon)\}_\varepsilon \subset B$ a family of initial data related to each $P_\varepsilon$ with solutions $\{S_\varepsilon(t) z_\varepsilon\}_\varepsilon$. Then there exists a constant $\hat{C}$ that  does not depend on $t,\varepsilon$ such that  
	\begin{equation*}
	E_\varepsilon(t) \leq \hat{C} \quad \mbox{ and } \quad 
	\|S_\varepsilon(t) z_\varepsilon\|_{\mathcal{H}_\varepsilon} \leq \hat{C}, \quad \forall \ \varepsilon,t >0.
	\end{equation*}
\end{lemma}

\begin{proof}
	From (\ref{18}), (\ref{up}), Remark \ref{remark1} and the fact that for each $\varepsilon$, $E_\varepsilon$ is decreasing, we have
	\begin{align*}
	K_2\|S_\varepsilon(t) z_\varepsilon\|_{\mathcal{H}_\varepsilon}   - K_3 	& \leq E_\varepsilon(t) \leq E_\varepsilon(0) \\
	&\leq K_1(\|u_\varepsilon \|_\varepsilon^2 + \|v_\varepsilon \|_2^2)^4 + K_3  \\
	& \leq K_1\hat{K}(B,\mu)+K_3
	\end{align*}
	where $\hat{K}(B,\mu)$ is a constant which depends only on $B$ and $\mu$, and $K_1, K_2$ and $ K_3$ do not depend on $t$ and $\varepsilon$.
\end{proof}

Now we are in position to state and prove the main results of this section.

\begin{theorem}[Singular limit]\label{singular}
	Under the assumptions of $(\ref{2})$-$(\ref{aa})$. Given a sequence $\{\varepsilon_n\} $ of positive numbers, let $(u_n(t), \partial_t u_n (t))$ be the weak solution to $P_{\varepsilon_n}$ with initial data $(v_0,v_1) \in \mathcal{H}_0$. 
	Then if $\varepsilon_n \to 0$ when $n \to \infty$, there exist a weak solution $(u(t),\partial_t u(t))$ of $P_0$ with the same initial data, such that for any $T>0$:
	\begin{align*}
	u_n &\overset{*}{\rightharpoonup} u \text{ in } L^{\infty}(0,T; (H_0^1(\Omega))^3),\\
	\partial_t u_n &\overset{*}{\rightharpoonup} \partial_t u \text{ in } L^{\infty}(0,T; (L^2(\Omega))^3).
	\end{align*}   
	
\end{theorem}

\begin{proof}
	Using Lemma \ref{lemma15} for $B=\{(v_0,v_1)\}$ and equation (\ref{up}), for some constant $K$,
	\begin{align}\label{36}
	\| (u_n(t), \partial_t u_n(t)) \|_{\mathcal{H}_0} \leq K.
	\end{align}
	Then, we have  for any $T>0$,
	\begin{align*}
	u_n &\overset{*}{\rightharpoonup} u \text{ in } L^{\infty}(0,T; (H_0^1(\Omega))^3),\\
	\partial_t u_n &\overset{*}{\rightharpoonup} \partial_t u \text{ in } L^{\infty}(0,T; (L^2(\Omega))^3).
	\end{align*} 	
Fixing $n$ and multiplying $P_{\varepsilon_n}$ by a function $\phi \in (H_0^1(\Omega))^3$, we get
	\begin{align}\label{33}
	\frac{d}{dt}\left\langle \partial_{t} u_n,\phi\right\rangle+\mu \left\langle \nabla u_n,\nabla \phi\right\rangle
	&+ \varepsilon_n  \left\langle \text{div}u_n, \text{div}\phi\right\rangle  \\
	&+ \alpha \left\langle \partial_t u_n, \phi\right\rangle + \left\langle f (u_n), \phi\right\rangle = \left\langle b,\phi\right\rangle. \notag
	\end{align} 
It is clear that
	\begin{align*}
	\left\langle \nabla u_n,\nabla \phi\right\rangle &\underset{n \to \infty}{\longrightarrow}  \left\langle \nabla u,\nabla \phi\right\rangle,\\
	\left\langle \partial_t u_n, \phi\right\rangle &\underset{n \to \infty}{\longrightarrow} \left\langle \partial_t u, \phi\right\rangle, \\
	\varepsilon_n  \left\langle \text{div}u_n, \text{div}\phi\right\rangle &\underset{n \to \infty}{\longrightarrow} 0 \,\,\,\text{ from } (\ref{36}).
	\end{align*}
	Additionally, 
	we have 
	\begin{align*}
	\left\langle f(u_n)-f(u),\phi\right\rangle &\leq \sum_{i=1}^3 \int_{\Omega} |f_i(u_n)-f_i(u)||\phi_i|\\
	&\leq \sum_{i=1}^3 \int_{\Omega} M_g(1 + \sum_{j=1}^3|u_n^j |^{p-1} + |u^j |^{p-1})|u_n-u| |\phi_i|.
	\end{align*} 
Then proceeding analogously to (\ref{23})-(\ref{Jota}) and using Simon's compactness theorem \cite{simon} 
we have that  
	\begin{align*}
	\|u_n-u\|_2, \|u_n-u\|_{\frac{6}{4-p}}  \to 0,
	\end{align*}
which implies
	\begin{align*}
	\left\langle f(u_n),\phi\right\rangle &\underset{n \to \infty}{\longrightarrow}  \left\langle f(u),\phi\right\rangle.
	\end{align*}
	Therefore (\ref{33}) converges to
	\begin{align}\label{34}
	\frac{d}{dt}\left\langle \partial_{t} u,\phi\right\rangle+\mu \left\langle \nabla u,\nabla \phi\right\rangle
	+ \alpha \left\langle \partial_t u, \phi\right\rangle + \left\langle f (u), \phi\right\rangle = \left\langle b,\phi\right\rangle,
	\end{align} 
	which means that $(u,\partial_t u)$ is a weak solution of $P_0$ and $u(0)=u_0$.
	
	Finally we multiply equations (\ref{33}) and (\ref{34}) by a test function $\psi \in H^1([0,T])$ such that $\psi(0)=1$,  $\psi(T)=0$  and integrating on $[0,T]$, we obtain for all $\phi \in (H_0^1(\Omega))^3$,	
	\begin{align*}
	&\int_0^T\frac{d}{dt}\left\langle \partial_{t} u_n,\phi\right\rangle \psi dt+\mu\int_0^T \left\langle \nabla u_n,\nabla \phi\right\rangle \psi dt + \varepsilon_n \lambda_0 \int_0^T \left\langle \text{div}u_n, \text{div}\phi\right\rangle \psi dt  \\
	&\hspace{3.2cm}\alpha \int_0^T\left\langle \partial_t u_n, \phi\right\rangle \psi dt + \int_0^T\left\langle f (u_n), \phi\right\rangle \psi dt= \int_0^T\left\langle b,\phi\right\rangle \psi dt,\\
	&\int_0^T\frac{d}{dt}\left\langle \partial_{t} u_,\phi\right\rangle \psi dt+\mu \int_0^T\left\langle \nabla u_,\nabla \phi\right\rangle \psi dt + \alpha \int_0^T\left\langle \partial_t u, \phi\right\rangle \psi dt + \int_0^T\left\langle f (u), \phi\right\rangle \psi dt \\
	&\hspace{9.3cm}= \int_0^T\left\langle b,\phi\right\rangle \psi dt.
	\end{align*}
Solving the integrals and taking $n \to \infty$,
	\begin{align*}
	&-\left\langle v_1,\phi\right\rangle - \int_0^T\left\langle \partial_{t} u,\phi\right\rangle \frac{d}{dt}\psi dt +\mu\int_0^T \left\langle \nabla u,\nabla \phi\right\rangle \psi dt + \alpha \int_0^T\left\langle \partial_t u, \phi\right\rangle \psi dt\\  
	&\hspace{7.0cm}+\int_0^T\left\langle f (u), \phi\right\rangle \psi dt = \int_0^T\left\langle b,\phi\right\rangle \psi dt,\\
	&-\left\langle \partial_{t} u(0),\phi\right\rangle - \int_0^T\left\langle \partial_{t} u,\phi\right\rangle \frac{d}{dt}\psi dt +\mu \int_0^T\left\langle \nabla u_,\nabla \phi\right\rangle \psi dt  + \alpha \int_0^T\left\langle \partial_t u, \phi\right\rangle \psi dt  \\
	&\hspace{7.0cm} +\int_0^T\left\langle f (u), \phi\right\rangle \psi dt = \int_0^T\left\langle b,\phi\right\rangle \psi dt.
	\end{align*}
	
	Therefore, $\partial_{t} u(0)=v_1$, which ends the proof.
\end{proof}

\begin{remark}
From Theorem \ref{singular} and its proof, it is worth making two comments as follows:
 \begin{itemize} 	
		\item the limit in the previous theorem is singular in the sense that $\mathcal{H}_{\varepsilon}$ is not the same for $\varepsilon$ varying the range $[0, 1]$;
		
		\item	the space $\mathcal{H}_{\varepsilon}$ for  weak solutions  is defined by means of  $(H_0^1(\Omega))^3 \times (L^2)^3$, where $H_0^1(\Omega)$ is provided with the norm $\|\cdot\|_{\varepsilon}$. Therefore, it makes sense to consider $(v_0,v_1)$ as  initial data to any $P_\varepsilon$.
	\end{itemize}
\end{remark}

\begin{theorem}[Upper semicontinuity]\label{upper} Under the assumptions   $(\ref{2})$-$(\ref{aa})$, the family of attractors $\{\mathcal{A}_\varepsilon\}$ is upper semicontinuous with restect $\varepsilon \to 0$.  More precisely, 
	\[  \lim_{\varepsilon \to 0 } d_{\mathcal{H}_0} (\hat{\i}_\varepsilon (\mathcal{A}_\varepsilon) , \mathcal{A}_0)=0, \]	    
	where $d_{\mathcal{H}_0}$ denotes Hausdorff semi-distance and $\hat{\i}_\varepsilon: \mathcal{H}_\varepsilon \rightarrow \mathcal{H}_0$ is the identity map.
\end{theorem}

\begin{proof}
The proof is done by contradiction arguments and follows similar lines  as presented e.g. in \cite{hale1988,geredeli,singular}.

\smallskip
Let us assume, for some $\epsilon>0$, that
 
	\[ \sup_{y \in \mathcal{A}_\varepsilon} \inf_{z \in \mathcal{A}_0} \|\hat{\i}_{\varepsilon}(y)-z \|_{\mathcal{H}_0} \geq \epsilon.  \]
	Since for any $\varepsilon$, $\mathcal{A}_{\varepsilon}$ is compact, there exists a sequence $\{y_n^0 \}_n$ such that $y_n^0 \in \mathcal{A}_{\varepsilon_n}$ and 
	\[ \inf_{z \in \mathcal{A}_0}\|\hat{\i}_{\varepsilon_n}(y_n^0)-z \|_{\mathcal{H}_0} \geq \epsilon. \]
Let $y_n(t)=(u_n(t), \partial_t u_n(t) )$ be a full trajectory in $\mathcal{A}_{\varepsilon_n}$ such that $y_n(0)=y_n^0$. 
From Lemma  \ref{Lemma4.1},
	\begin{align}\label{35}
	\|y_n(t)\|_{\mathcal{H}_0} \leq  R_1 +1.
	\end{align}
Also, from Theorem \ref{3.16}, for each $n \in \mathbb{N}$, there exists $R_2^{\varepsilon_n}> 0$ such that
\[
	\|\partial_t y_n(t)\|_{\mathcal{H}_{0}} \leq \| \partial_t y_n(t)\|_{\mathcal{H}_{\varepsilon_n}} \leq R_2^{\varepsilon_n} .
\]
Additionally, from (\ref{up}), we obtain the existence of $R_2>0$, that does not depend on $\varepsilon_n$ for all $n$, such that
	\[\|\partial_t y_n(t)\|_{\mathcal{H}_{0}} \leq \|\partial_t y_n(t)\|_{\mathcal{H}_{\varepsilon_n}} \leq R_2, \quad \forall \ t, n. \]
In this way, one sees that	
$$
\mathcal{E}_{\varepsilon_n} u =- \alpha \partial_t  u - f(u)  -\partial_{tt} u +h  \in (L^2(\Omega))^3.
$$
Thus, multiplying this  identity by $\mathcal{E}_{\varepsilon_n} u$, integrating and using H\"older's inequality, there exists $R_3 > 0$, not depending on $\varepsilon_n$, such that
\begin{equation*}
	\|\mathcal{E}_{ \varepsilon_n} u (t)\|_2 \leq R_3, \quad \forall \ t, n,
	\end{equation*}
from where it follows that 
	\begin{align*}
	(y_n) &\text{  is bounded on  } L^\infty(\mathbb{R},\mathcal{H}_0 ), \\
	(\partial_t y_n) &\text{  is bounded on  } L^\infty(\mathbb{R},\mathcal{H}_0 ).
	\end{align*}
Using Simon's Theorem of compactness for the spaces $\mathcal{H}_0^1 \overset{c}{\hookrightarrow} \mathcal{H}_0 \hookrightarrow \mathcal{H}_0$, we have that for any $T>0$,
	there exists a subsequence $\{y_{n_l}\}$ and $y \in C([-T,T], \mathcal{H}_0 )$ such that  
	\begin{align*}
	\lim_{l \to \infty}\sup_{t \in [-T,T]} \|y_{n_l}(t)-y(t)  \|_{\mathcal{H}_0} =0.
	\end{align*}
	In particular,
	\begin{align*}
	\lim_{l \to \infty} \|\hat{\i}_{\varepsilon_{n_l} }( y_{n_l}^0)-y(0)  \|_{\mathcal{H}_0} =0.
	\end{align*}
	In order to get the desired contradiction, it remains to prove $y(0) \in \mathcal{A}_0$. In fact, since $\{ y_{n_l}^0\}_l$ is bounded on $\mathcal{H}_0$, we can process as in the proof of Theorem \ref{singular}, and prove that $y$ is a solution of $P_0$ for time varying $t \in [-T,T]$ with initial data $y(0)$. Since $T>0$ is arbitrary and (\ref{35}) holds true, then $y(t)$ is a bounded full trajectory of $P_0$. This implies that $y(0)\in \mathcal{A}_0$. The proof Theorem \ref{upper} is complete.
	\end{proof}

\section*{Acknowledgment}   
\noindent L. E. Bocanegra-Rodr\'{\i}guez was supported by CAPES, finance code 001 (Ph.D. Scholarship).
M. A. Jorge Silva was partially supported by Funda\c{c}\~ao Arauc\'aria grant 066/2019 and CNPq grant 301116/2019-9. 
T. F. Ma was partially supported by CNPq grant 312529/2018-0 and FAPESP grant 2019/11824-0.
P. N. Seminario-Huertas was partially supported by INCTMat-CNPq and CAPES-PNPD.

\bigskip

\noindent {\bf Email addresses:} 

L. E. Bocanegra-Rodr{\'i}guez:  {lito@icmc.usp.br} 

M. A. Jorge Silva: {marcioajs@uel.br} 

T. F. Ma: {matofu@mat.unb.br} 

P. N. Seminario-Huertas: {pseminar@icmc.usp.br}

\end{document}